\documentclass[12pt, reqno]{amsart}
\usepackage{amssymb}
\usepackage[mathscr]{eucal}
\usepackage{graphicx}
\usepackage{float}
\usepackage{subfigure}
\usepackage{xspace}
\usepackage{tikz}

\usepackage[pagewise]{lineno}

\newcommand{\p}{\partial}
\newcommand{\dd}{\mathrm{d}}

\newcommand{\R}{\mathbb{R}}
\newcommand{\ti}{\textsf{I}_\Omega}
\newcommand{\tiw}{\textsf{I}_{W}}

\newcommand{\I}{\mathcal{I}}
\newcommand{\D}{\mathbb{\D}}
\newcommand{\n}{{\mathbf n}}

\newcommand{\LL}{\mathcal{L}}

\newcommand{\B}{{\mathcal B}}

\newcommand{\del}{{\partial}}

\newcommand{\tr}{\tilde \rho}

\newcommand{\tx}{\tilde{x}}
\newcommand{\ty}{\tilde{y}}
\newcommand{\tE}{\widetilde{E}}
\newcommand{\tu}{\tilde{u}}
\newcommand{\tv}{\tilde{v}}
\newcommand{\tp}{\tilde{p}}

\theoremstyle{plain}
\newtheorem{theorem}{Theorem}[section]

\newtheorem{lemma}{Lemma}[section]

\theoremstyle{definition}
\newtheorem{definition}{Definition}[section]
\theoremstyle{remark}

\newtheorem{remark}{Remark}[section]
\numberwithin{equation}{section}

\begin{document}
\title[Hypersonic limit of Euler flows passing a wedge] {Hypersonic limit of two-dimensional steady compressible Euler flows passing a straight wedge}

\author{Aifang Qu}
\author{Hairong Yuan}
\author{Qin Zhao}

\address[A. Qu]
         {Department of Mathematics, Shanghai Normal University,
Shanghai,  200234,  China} \email{\tt afqu@shnu.edu.cn, aifangqu@163.com}

\address[H. Yuan]{School of Mathematical Sciences and Shanghai Key Laboratory of Pure Mathematics and Mathematical Practice,
East China Normal University, Shanghai
200241, China}
\email{\tt hryuan@math.ecnu.edu.cn}

\address[Q. Zhao]{School of Mathematical Sciences, Shanghai Jiao Tong University, Shanghai 200240, China}
\email{\tt zhao@sjtu.edu.cn}

\date{\today}

\begin{abstract}
We formulated a problem on hypersonic limit of two-dimensional steady non-isentropic compressible Euler flows  passing a straight wedge. It turns out that  Mach number of the upcoming uniform supersonic flow increases to infinite may be taken as the adiabatic exponent $\gamma$ of the polytropic gas decreases to $1$. We proposed a form of the Euler equations which is valid if the unknowns are measures and constructed a measure solution contains Dirac measures supported on the surface of the wedge. It is proved that as $\gamma \to1$, the sequence of solutions of the compressible Euler equations that containing a shock ahead of the wedge converge vaguely as measures to the measure solution we constructed.  This justified the Newton theory of hypersonic flow passing obstacles in the case of two-dimensional straight wedges. The result also demonstrates the necessity of considering general measure solutions in the studies of boundary-value problems of systems of hyperbolic conservation laws.
\end{abstract}
\maketitle

\tableofcontents


\section{{Introduction}}\label{S:1}
We are concerned with steady non-isentropic
compressible Euler flows passing an obstacle.
It is important to understand what happens if the Mach number of the upcoming supersonic flows goes to infinity, as required, for example, by the designation of hypersonic aircraft ({\it cf.} \cite[chapter 1]{A}. From mathematical point of view, it is also interesting, since there appears some new singularity which requires rigorous justification. It also indicates the
indispensable role of measure solutions in the studies of systems of conservation laws.

It has been discovered for a long time that for some conservation laws, the Riemann problems cannot be solved in the sense of integral weak solutions that are measurable functions with respect to the Lebesgue measure (see \cite{YZ} and references cited therein). Singular measure (with respect to the Lebesgue measure), such as Dirac measure, was introduced to solve such equations. For example, it is well-known that concentration of mass, $i.e.$, $\delta$-shocks, is notably involved in the solution of the pressureless Euler equations \cite{CL}, and well-posedness for initial data being Radon measures was established in \cite{HW} for one-space-dimensional case. For compressible Euler equations of Chaplygin gases,  $\delta$-shocks
are also necessary to solve certain Riemann problems (see, for example, \cite{KW,SWY}). However, one does not know whether such singular solutions are necessary for boundary  value problems of the compressible Euler equations of polytropic gases, even for a physically significant limiting case. In this paper, we will show that the Newton theory (see, for example, \cite[section 3.2]{A}, \cite[chapter 3]{HP}, \cite[section 5.2]{R}, or \cite{LO}) of hypersonic flow passing a wedge can be explained rigorously as measure solutions of a boundary value problem of the steady compressible Euler equations when the adiabatic exponent of the polytropic gases decreases to $1$. Also, to our knowledge, considering flows in  domains with boundaries,  there is no analysis on singular measure solutions in previous works.

We believe measure solutions will play an important role in the studies of multidimensional hyperbolic conservation laws. Firstly, as shown above, it appears in significant physical problems. Secondly,  it is strategically reasonable to construct a measure solution to the multi-dimensional compressible Euler equations, as done in \cite{CSW} by applying the theory of optimal transport and gradient flows, and then study its regularity later, for example, showing no singular measure appears in the solution in some cases. Thirdly, as motivated by \cite{H,HZ}, once we justified the formation of singular measures, as done in this paper, one may use such singular solutions as approximate solutions to study some difficult problems, such as supersonic flows past obstacles studied in \cite{CZZ,WZ}. This paper represents the first step in this direction.

For the problem of hypersonic limit, intuitively, when the Mach number of the upcoming supersonic flow goes to infinity, it looks that the pressure of the gas is vanishing, and the limit flow may be corresponding to the pressureless flow, a case similar to the one studied by Chen and Liu in \cite{CL}, where they  proved that $\delta$-shocks may form in the vanishing pressure limit of  solutions containing two shocks to the Euler equations for isentropic fluids in one-dimensional Riemann problem, and it is indeed a measure solution of the pressureless flow. See \cite{YZ} for more related results.
For the case of two dimensional steady pressureless Euler flows, $\delta$-shocks have been constructed in \cite{ZZ}. However,
in this work, we found that hypersonic limit is not the vanishing pressure limit, and the singularity is not a $\delta$-shock.    This is also true for the higher Mach number limit for the one-dimensional piston problem \cite{QZY}.

We remark by passing that the readers shall not be confused by the measure solutions we called in this paper (or measure-valued solutions called in \cite{CSW}) with measure-valued solutions proposed by DiPerna \cite{Diperna}, which was intensively studied recently (see, for example, \cite{BDS,BF} and references therein). Our measure solutions are measures supported on the physical space of independent variables, which describe, for example, concentration of mass etc. on a space-time set. The measure-valued solution of DiPerna is a measure supported  on the phase space of dependent variables, and reflects uncertain of taking values of the unknowns in the phase space at a space-time point. Therefore the interpretations of Euler equations are quite different ({\it cf.} Definition \ref{def32}).

The rest of the paper is divided into three sections.  In \S \ref{sec2}, we review the problem of supersonic flows passing a wedge and present the solutions obtained by shock polar. We show  the hypersonic limit is actually the limit that $\gamma\to 1$ and present a key lemma. In \S \ref{sec3}, we propose a form of the two-dimensional steady compressible Euler equations and construct its measure solutions. The main difficulty is how to understand the Euler equations when the unknowns are measures, and provide the appropriate concept of solutions. Finally, in \S \ref{sec4}, we show the sequence of integral weak solutions of the Euler equations obtained by shock polar converge vaguely as measures to the limit singular measure solution we constructed, and consistency holds for the general form of Euler equations we presented. This justified the hypersonic limit rigorously and  provide a mathematical justification of the Newton theory. The main result of this paper is summarized at the end of the paper, as Theorem \ref{thm41}.

\section{The hypersonic limit problem}\label{sec2}
In this section we formulate the hypersonic limit problem of uniform supersonic flows passing a straight wedge. Some results of this section are well-known. However, for completeness and easier reading, we shall present them with some details.

\subsection{The problem of compressible Euler flows passing a wedge}\label{sec21}
We consider the following two-dimensional steady non-isentropic compressible Euler system for polytropic gases, consisting of conservation of mass, momentum and energy:
\begin{equation}\label{eq21}
\begin{cases}
\displaystyle \del_x(\rho u)
  + \del_y \left(\rho v\right)=0,\\[8pt]
\displaystyle \del_x(\rho u^2+p)
  + \del_y \left(\rho u v\right)=0,\\[8pt]
\displaystyle \del_x\left(\rho uv\right)
  + \del_y \left(\rho v^2+p \right) = 0,\\[8pt]
  \displaystyle \del_x\left(\rho uE\right)
  + \del_y \left(\rho vE\right) = 0.
\end{cases}
\end{equation}
Here $\rho$,  $p$, and $(u, v)$ represent respectively  the density of mass, scalar pressure and velocity of the flow, and $$E=\frac{1}{2}(u^2+v^2)+\frac{\gamma}{\gamma-1}\frac{p}{\rho}$$
is the total energy per unit mass, with $\gamma>1$ the adiabatic exponent. Recall that for polytropic gas, the constitutive
relations are given by
\[
p=\kappa \rho^{\gamma}\exp(S/c_v), \qquad
T_e=\frac{p}{(\gamma-1)c_v\rho},
\]
where $S$ and $T_e$ are the entropy and  temperature of the flow respectively, and $\kappa, c_{v}$ are positive constants.

System \eqref{eq21} can be written in the general form of conservation laws:
\begin{equation}\label{eq22}
\del_x F(U)+\del_y G(U)=0, \quad U=(\rho, u, v, E)^{\top},
\end{equation}
where
\[
\begin{split}
&F(U)=\left(\rho u,\; \rho u^2+p, \;\rho uv,\; \rho uE\right)^\top, \\[5pt]
&G(U)=\left(\rho v,\;\rho uv,\; \rho v^2+p,\; \rho
vE\right)^\top,
\end{split}
\]
and
\begin{equation}\label{eq23}
p=(E-\frac12(u^2+v^2))\frac{\gamma-1}{\gamma}\rho.
\end{equation}
Recall that for polytropic gas, the local sound speed is given by
$c=\sqrt{{\gamma p}/{\rho}}$, and Mach number is defined by  $M=\sqrt{u^2+v^2}/c$. For $M>1$, the flow is called supersonic and it is well-known that \eqref{eq21} is then a hyperbolic system of conservation laws.

The solid wedge (or a ramp) is given by $\{(x,y)\in\R^2: x\ge0, 0\le y\le ax\}$, and $a=\tan\theta>0$, with $\theta\in(0,\frac{\pi}{2})$ the opening angle of the wedge. Therefore, the domain occupied by gas is
$$\Omega=\{(x,y)\in\R^2: x>0, y>ax\}.$$
On the surface of the wedge
$$W=\{(x,y)\in\R^2: x\ge0, y=ax\},$$
we propose the slip condition
\begin{eqnarray}\label{eq24}
v=au \qquad\text{on}\ \ W.
\end{eqnarray}
On the line $\I=\{(x,y)\in\R^2: x=0, y>0\}$,
the flow is given:
\begin{equation}\label{eq25}
U=U_\infty\qquad\text{on}\ \ \I,
\end{equation}
and $U_\infty=(\rho_\infty, u_\infty, 0, E_\infty)^\top$ is constant.

The solution to problem \eqref{eq22}--\eqref{eq25} is well-known if $M_\infty=\sqrt{\frac{\rho_\infty u_\infty^2}{\gamma p_\infty}}>1$ and we will present it in \S \ref{sec22}. What we are interested is what happens for the hypersonic limit \begin{equation}\label{eq26}M_\infty\to \infty.\end{equation}

To understand \eqref{eq26}, we need to introduce the following non-dimensional independent and dependent variables:
\begin{eqnarray*}
&&\tx=x/l,\qquad \ty=y/l;\\
&& \tr=\rho/\rho_\infty,\ \ \tu=u/u_\infty, \ \ \tv=v/u_\infty,\ \ \tilde{S}=S/c_v,\\
&& \tp=p/(\rho_\infty u_\infty^2),  \ \ \ \tE=E/u_\infty^2.
\end{eqnarray*}
Here $l>0$ is a constant with dimension of length.
Direct computation yields
$$\tp=\frac{1}{\gamma M_\infty^2}\exp(\tilde{S})\tr^\gamma,$$
and $\tilde{c}=\sqrt{\gamma \tp/\tr}=c/u_\infty$, and it is easy to check that $\tilde{U}=(\tr, \tu, \tv, \tE)^\top$ still solves \eqref{eq22} and \eqref{eq23}, with $x, y$ there replaced by $\tx, \ty$. So without loss of generality, {\it in the rest of this paper we still write $\tilde{U}$ as $U$, and $(\tx, \ty)$ as $(x,y)$}. The domain and its boundary, as well as boundary condition \eqref{eq24} are also {\it not} changed. The initial data \eqref{eq25} becomes
\begin{eqnarray}\label{eq27}
U=U_0=(1,1,0, E_0)^\top,
\end{eqnarray}
with $E_0>1/2$ being a fixed constant. Therefore, there are only three parameters, namely $\theta\in(0,\pi/2), \gamma>1, E_0>1/2$ to determine the problem of  supersonic flow passing a wedge.

Applying \eqref{eq23} with the data \eqref{eq27}, one has
\begin{equation}\label{eq28}
\frac{1}{M_0^2}=c_0^2=(\gamma-1)(E_0-\frac12).
\end{equation}
So we conclude that:
 {\begin{center}{\it For fixed total energy per unit mass,\\ the hypersonic limit $M_0\to\infty$ is equivalent to the limit $\gamma\downarrow 1$.}\end{center}}
\noindent This is the case that we are mostly interested in this paper. For the other case, {\it i.e.}, $\gamma>1$ being fixed and $E_0\downarrow1/2$, see Remark \ref{rm21} below.

Since $c_0=\sqrt{\gamma p_0/\rho_0}$, from \eqref{eq28}, we also infer that
\begin{equation}\label{eq29}
p_0=\frac{\gamma-1}{\gamma}(E_0-\frac12).
\end{equation}
So at first glance the hypersonic limit $\gamma\to1$ is also the vanishing pressure limit. However, this is not true as we will see later.

For simplicity of writing, from now on we set
\begin{equation}\label{eq210}
\epsilon=\gamma-1, \qquad E_0'=E_0-\frac12.
\end{equation}

\subsection{Existence of weak solutions containing shocks}\label{sec22}

We now review some results on solutions of problem \eqref{eq22}-\eqref{eq24}\eqref{eq27}. It is known that the solution is not continuous and shocks will appear. Therefore we need the following concept.
\begin{definition}[integral entropy solutions]\label{def21}
We say $U\in L^\infty(\Omega)$ is an integral entropy solution to problem \eqref{eq22}-\eqref{eq24}\eqref{eq27}, if for any $\phi\in C_0^1(\R^2)$ ({\it i.e.} continuously differentiable functions with compact supports in $\R^2$), there holds
\begin{eqnarray}\label{eq211}
&&\int_\Omega (F(U)\p_x\phi+G(U)\p_y\phi)\,\dd x\dd y\nonumber\\
&=&\int_0^\infty (aF(U|_W)-G(U|_W))\phi(x,ax)\,\dd x-\int_0^\infty F(U_0)\phi(0,y)\,\dd y,
\end{eqnarray}
and furthermore, the pressure increases if the flow passing across any discontinuities in $U$.
\end{definition}

Problem \eqref{eq22}-\eqref{eq24}\eqref{eq27} is a Riemann problem with boundary conditions and one can construct a solution $U(x,y)=V(x/y)$. Set $\eta=x/y$.
Suppose the solution is piecewise constant, of the form
\begin{equation*}\label{eq216add}
V(\eta)=\begin{cases}
V_0=(1,1,0,E_0)^\top,& 0\le\eta<\frac{1}{\sigma},\\
V_1,& \frac{1}{\sigma}<\eta\le\frac1a.
\end{cases}\end{equation*}
Then to fulfill \eqref{eq211}, by taking $\phi$ supported near the discontinuity $\eta=1/\sigma$, $V_1$ and $\sigma$ shall satisfy the following Rankine-Hugoniot conditions:
\begin{eqnarray}\label{eq216}
(F(V_1)-F(V_0))-\frac{1}{\sigma}(G(V_1)-G(V_0))=0.
\end{eqnarray}
From this one can solve that (see, for example, \cite[pp.280-282]{RG})
\begin{eqnarray*}
&&\frac{\rho_1}{\rho_0}=\frac{(\gamma+1)M_0^2\sin^2\alpha}{2+(\gamma-1)M_0^2\sin^2\alpha},\quad\frac{p_1}{p_0}=\frac{2\gamma M_0^2\sin^2\alpha-(\gamma-1)}{\gamma+1},\\
&&\sigma=\tan\alpha=\frac{u_0-u_1}{v_1},\quad
v_1^2=\frac{(u_0-u_1)^2(u_0u_1-V_*^2)}{V_*^2+\frac{2}{\gamma+1}u_0^2-u_0u_1}, \\
&&\frac{1}{M_0^2}=\frac{\gamma+1}{2}\frac{V_*^2}{u_0^2}-\frac{\gamma-1}{2},
\end{eqnarray*}
where $\alpha$ is the angle between the shock-front and the $x$-axis.  Substitute $\gamma=\epsilon+1$, $M_0^2=1/(\epsilon E_0')$, and $u_0=1, \rho_0=1$, $p_0=\frac{\epsilon}{\epsilon+1}E_0'$ into these formulas, we get
\begin{eqnarray}
&&\rho_1^\epsilon=\frac{\epsilon+2}{\epsilon}\frac{\sin^2\alpha^\epsilon}{2E_0'+\sin^2\alpha^\epsilon},\label{eq217}\\
&&p_1^\epsilon=\frac{2(\epsilon+1)\sin^2\alpha^\epsilon-\epsilon^2E_0'}{(\epsilon+1)(\epsilon+2)},\label{eq218}\\
&& {(v_1^\epsilon)}^2=\frac{(1-u_1^\epsilon)^2(u_1^\epsilon-\frac{\epsilon}{\epsilon+2}-\frac{2\epsilon}{\epsilon+2}E_0')}{1-u_1^\epsilon+\frac{2\epsilon}{\epsilon+2}E_0'},\label{eq219}\\
&&\sigma^\epsilon=\tan\alpha^\epsilon=\frac{1-u_1^\epsilon}{v_1^\epsilon}.\label{eq220}
\end{eqnarray}
The above Rankine-Hugoniot conditions also imply that
\begin{equation}\label{eq221}
E_1^\epsilon=E_0.
\end{equation}

The formula \eqref{eq219} with fixed $\epsilon>0$ represents a curve in the $(u, v)$-plane, which is called {\it shock polar}. Using the boundary condition \eqref{eq24}, for each fixed $\epsilon>0$, one may uniquely solve $V_1^\epsilon=(\rho_1^\epsilon, u_1^\epsilon, v_1^\epsilon=au_1^\epsilon, E_0)^\top$ which satisfies $p_1^\epsilon>p_0.$ We thus obtain a sequence of integral entropy solutions
\begin{equation}\label{eq222}
U^\epsilon(x,y)=V^\epsilon(\frac{x}{y})=\begin{cases}
V_0=(1,1,0,E_0)^\top,& 0\le\frac{x}{y}<\frac{1}{\sigma^\epsilon},\\
V_1^\epsilon,& \frac{1}{\sigma^\epsilon}<\frac{x}{y}\le\frac1a
\end{cases}\end{equation}
to problem \eqref{eq22}-\eqref{eq24}\eqref{eq27}, for all $\epsilon>0$.

\subsection{Point-wise limits of weak solutions}\label{sec23}
As mentioned above, we need to understand $\lim_{\epsilon\to 0}U^\epsilon$.
\begin{lemma}\label{lem22}
For $\epsilon>0$ small, $\rho_1^\epsilon, u_1^\epsilon, v_1^\epsilon, \sigma^\epsilon$ are all $C^2$ with respect to $\epsilon$, and
\begin{eqnarray}
&&\lim_{\epsilon\to 0}u_1^\epsilon=\cos^2\theta,\quad\lim_{\epsilon\to 0}v_1^\epsilon=\cos\theta\sin\theta,\label{eq223}\\
&& \lim_{\epsilon\to 0}p_1^\epsilon=\sin^2\theta,\label{eq224}\\
&& \lim_{\epsilon\to 0}\epsilon\rho_1^\epsilon=\frac{2\sin^2\theta}{2E_0'+\sin^2\theta},\label{eq225}\\
&& \sigma^\epsilon-a=\frac{\frac12\sin^2\theta+E_0'}{\cos^3\theta\sin\theta}\epsilon+o(\epsilon),\label{eq226}\\
&&\lim_{\epsilon\to 0}\rho_1^\epsilon(\sigma^\epsilon-a)=\frac{\sin\theta}{\cos^3\theta}.\label{eq227}
\end{eqnarray}
\end{lemma}

\begin{proof}
1. Note that the entropy condition requires that $u=u_1<1$.  We now set $\lambda=\lambda(\epsilon)=\frac{\epsilon}{\epsilon+2}$ and ({\it cf.} \eqref{eq219})
$$H(u,\lambda)=(1-u)^2(u-\lambda-2\lambda E_0')-a^2u^2(1-u+2\lambda E_0').$$
The equation $H(u, 0)=0$ has roots $u=1, u=0$ and $u=\frac{1}{1+a^2}=\cos^2\theta.$ By physical considerations, we are only interested in the root $u(0)=\cos^2\theta$ in this paper.

2. Direct computation shows that $$\p_uH(u(0),0)=-\sin^2\theta<0.$$ So the implicit function theorem yields that we can solve a function $u=u(\lambda)\in C^2$  so that $H(u(\lambda), \lambda)=0$ for $\lambda>0$ small. Since $\p_\lambda H(u(0),0)=-\sin^2\theta(\sin^2\theta+2E_0')$, and $\lambda'(0)=1/2$, we get a $C^2$ function $u_1(\epsilon)=u(\lambda(\epsilon))$ for $\epsilon\ge0$ small, and
\begin{eqnarray}\label{eq228}
u_1(\epsilon)-u_1(0)=-(E_0'+\frac12\sin^2\theta)\epsilon+o(\epsilon).
\end{eqnarray}
This implies \eqref{eq223} with $u_1^\epsilon=u_1(\epsilon)$.

3. Then from \eqref{eq220}, we infer that  $\sigma^\epsilon$ and $\alpha^\epsilon$ are $C^2$ for $\epsilon\ge 0$, and particularly,
\begin{eqnarray}\label{eq229}
\lim_{\epsilon\to0}\sigma^\epsilon=\tan\theta=a,\quad \lim_{\epsilon\to0}\alpha^\epsilon=\theta.
\end{eqnarray}
The formulas \eqref{eq224}\eqref{eq225} follow directly from \eqref{eq217}\eqref{eq218}.

4. By \eqref{eq220},
\begin{eqnarray*}
\sigma^\epsilon-a=\frac1a\left(\frac{1}{u_1^\epsilon}-\frac{1}{\cos^2\theta}\right),
\end{eqnarray*}
so $$\sigma'(0)=-\frac1a\frac{1}{\cos^4\theta}u_1'(0)=\frac{E_0'+\frac12\sin^2\theta}{\sin\theta\cos^3\theta}$$ by \eqref{eq228}. This proves \eqref{eq226}, and \eqref{eq227} follows from \eqref{eq225} and \eqref{eq226}.
\end{proof}

The formula \eqref{eq224} is called as {\it Newton's sine-squared pressure law} (see (3.1.1) in \cite[p. 132]{HP}, or (3.3) in \cite[section 3.2]{A}). From \eqref{eq224} we see the hypersonic limit is different from the vanishing pressure limit. As the Mach number $M_0$ increases to infinity, the shock-front approaches the surface of the wedge, with distance of the order $\epsilon=\gamma-1$, or $1/M_0^2$. Since the support of $V_1^\epsilon$ becomes narrower and narrower, and the density blows up with an order $1/(\gamma-1)$, the above description of hypersonic limit is not sufficient. We need to introduce singular measure solutions to Euler equations.

\begin{remark}\label{rm21}
For the case that $\epsilon>0$ being fixed and $\delta=E_0'\downarrow0$, from \eqref{eq217}-\eqref{eq221}, where we write $U^\epsilon, \alpha^\epsilon, \sigma^\epsilon$ to be $U^\delta, \alpha^\delta, \sigma^\delta$ to indicate the dependance of the solutions on $\delta=E_0'$, we have
\begin{eqnarray*}
&&\lim_{\delta\to 0}\rho_1^\delta=\frac{\epsilon+2}{\epsilon},\quad
\lim_{\delta\to 0}\alpha^\delta=\alpha^0,\quad
\lim_{\delta\to 0}u_1^\delta=u_1^0,\quad \lim_{\delta\to 0}v_1^\delta=v_1^0,\\
&&\lim_{\delta\to 0}\sigma^\delta=\sigma^0=\tan\alpha^0=\frac{1-u_1^0}{v_1^0},\\
&&\lim_{\delta\to 0}p_1^\delta=\frac{2\sin^2\alpha^0}{\epsilon+2},
\end{eqnarray*}
and $(u_1^0, v_1^0)$ lies on the limiting shock polar 
$$\left(u_1^0-\frac{\epsilon+1}{\epsilon+2}\right)^2+(v_1^0)^2=\left(\frac{1}{\epsilon+2}\right)^2,$$
which is a circle. So for $\theta\in(0, \arcsin(1/(\epsilon+1)))$, we may solve two points $(u_1^0, v_1^0)$ on the circle so that $v_1^0=(\tan\theta) u_1^0$, and determine the corresponding shock-front angles $\alpha^0.$ Hence there is no concentration in this case of hypersonic limit. In the rest of this paper, we always consider the more challenge case that $\delta>0$ fixed, while $\epsilon\downarrow 0.$
\end{remark}

\section{Measure solutions and Euler equations}\label{sec3}

Let $\B$ be the Borel $\sigma$-algebra of the Euclidean plane $\R^2$. In this section we always consider Radon measures on $(\R^2, \B)$, and write
$$\langle m, \phi\rangle=\int_{\R^2}\phi(x,y)m(\dd x\dd y)$$
for the pairing between a Radon measure $m$ and  a  test function $\phi\in C_0(\R^2)$.
The standard Lebesgue measure of $\R^d$ is denoted by $\LL^d$ for $d\in\mathbb{N}$. A measure $\mu$ is absolute continuous with respect to a nonnegative measure $\nu$ is denoted by $\mu\ll\nu.$ The Dirac measure supported on a curve, which is singular to $\LL^2$, is defined as below ({\it cf.} \cite{CL}).

\begin{definition}[weighted Dirac measure supported on a curve]\label{def31}
Let $L$ be a Lipschitz curve given by $x=x(t), y=y(t)$ for $t\in[0,T)$, and $w_L(t)\in L_{\mathrm{loc}}^1(0,T)$. The Dirac measure supported on $L\subset\R^2$ with weight $w_L$ is defined by
\begin{eqnarray}\label{eq31}
\langle w_L\delta_L, \phi\rangle=\int_0^Tw_L(t)\phi(x(t), y(t))\sqrt{x'(t)^2+y'(t)^2}\,\dd t,\qquad \forall \phi\in C_0(\R^2).
\end{eqnarray}
\end{definition}

As an example, consider a vector field $(f, g)$ on $\R^2$ which is $C^1$ on $\R^2\setminus L$. Then for any $\phi\in C_0^1(\R^2)$, by Green theorem, one has
\begin{eqnarray}\label{eq32}
\int_{\R^2}(f\p_x\phi+g\p_y\phi)\,\dd x\dd y=\langle ([[f]],[[g]])\cdot \n\delta_L,\phi\rangle-\int_{\R^2\setminus L}(\p_xf+\p_yg)\phi\,\dd x\dd y.
\end{eqnarray}
Here $\n$ is the unit normal vector of $L$ obtained by rotating the tangent vector of $L$ clockwise with an angle $\pi/2$, and $[[f]](x,y)=\lim_{h\downarrow0}(f((x,y)-h\n)-f((x,y)+h\n))$. So Rankine-Hugoniot conditions guarantee that there is no concentration of mass, momentum and energy along a shock-front.

\subsection{A general formulation of two-dimensional steady compressible Euler equations}\label{sec31}

We now reformulate problem \eqref{eq22}-\eqref{eq24}\eqref{eq27} so that the unknowns may be measures.

\begin{definition}\label{def32}
For fixed $\epsilon\ge 0$, let $m^0,m^1, m^2, m^3, n^0,n^1, n^2, n^3, \wp$ be Radon measures on $\overline{\Omega}$, and $w^1_p, w^2_p$ locally integrable functions on $\R^+\cup\{0\}$.  Suppose that
\begin{itemize}
\item[i)] For $\n=(a,-1)/\sqrt{1+a^2}$ being the outward unit normal vector on $W$ (the surface of the wedge), one has
    \begin{eqnarray}\label{eq33}
    (w^1_p, w^2_p)\parallel \n\quad \text{or}\quad w^1_p+aw^2_p=0\qquad\LL^1\text{-a.e.};
    \end{eqnarray}
\item[ii)] For any $\phi\in C_0^1(\R^2)$, there hold
\begin{eqnarray}
\langle m^0, \p_x\phi\rangle+ \langle n^0, \p_y\phi\rangle+\int_0^\infty(\rho_0u_0)\phi(0,y)\,\dd y=0,&&\label{eq34}\\
\langle m^1, \p_x\phi\rangle+ \langle n^1, \p_y\phi\rangle+\langle \wp, \p_x\phi\rangle+\langle
w^1_p\delta_W,\phi \rangle+\int_0^\infty(\rho_0u_0^2+p_0)\phi(0,y)\,\dd y=0,&&\label{eq35}\\
\langle m^2, \p_x\phi\rangle+ \langle n^2, \p_y\phi\rangle+\langle \wp, \p_y\phi\rangle+\langle
w^2_p\delta_W,\phi \rangle+\int_0^\infty(\rho_0u_0v_0)\phi(0,y)\,\dd y=0,&&\label{eq36}\\
\langle m^3, \p_x\phi\rangle+ \langle n^3, \p_y\phi\rangle+\int_0^\infty(\rho_0u_0E_0)\phi(0,y)\,\dd y=0;&&\label{eq37}
\end{eqnarray}

\item[iii)] There is a nonnegative Radon measure $\varrho$ so that $\wp\ll\varrho$,  $(m^0, n^0)\ll \varrho$, $(m^k, n^k)\ll (m^0, n^0)$ ($k=1,2,3$), with derivatives
\begin{eqnarray}\label{eq38}
u=\frac{m^0(\dd x\dd y)}{\varrho(\dd x\dd y)}\quad\text{and}\quad v=\frac{n^0(\dd x\dd y)}{\varrho(\dd x\dd y)}
\end{eqnarray}
satisfy $\varrho$-a.e. that
\begin{eqnarray}
&&u=\frac{m^1(\dd x\dd y)}{m^0(\dd x\dd y)}=\frac{n^1(\dd x\dd y)}{n^0(\dd x\dd y)},\label{eq39}\\
&&
v=\frac{m^2(\dd x\dd y)}{m^0(\dd x\dd y)}=\frac{n^2(\dd x\dd y)}{n^0(\dd x\dd y)},\label{eq310}
\end{eqnarray}
and there is a $\varrho$-a.e. function $E$ so that
\begin{eqnarray}
E=\frac{m^3(\dd x\dd y)}{m^0(\dd x\dd y)}=\frac{n^3(\dd x\dd y)}{n^0(\dd x\dd y)};\label{eq311}
\end{eqnarray}
\item[iv)] If $\varrho \ll \LL^2$ with derivative $\rho(x,y)$,  and $\wp \ll \LL^2$ with derivative $p(x,y)$, in a neighborhood of $(x,y)\in\Omega$, then $\LL^2$-a.e. there holds
    \begin{eqnarray}\label{eq312}
    p=\frac{\epsilon}{\epsilon+1}\rho(E-\frac12(u^2+v^2));
    \end{eqnarray}
    Furthermore, classical entropy condition is valid for discontinuities of functions $\rho, u, v, E$ in this case.
\end{itemize}
Then we call $(\varrho, u, v, E)$ a measure solution to problem \eqref{eq22}-\eqref{eq24}\eqref{eq27}.
\end{definition}

We may write \eqref{eq34}-\eqref{eq37} formally as
\begin{eqnarray*}\begin{cases}
\p_xm^k+\p_yn^k=0,\quad k=0,3,\\
\p_xm^1+\p_yn^1+\p_x\wp=w^1_p\delta_W,\\
\p_xm^2+\p_yn^2+\p_y\wp=w^2_p\delta_W.
\end{cases}\end{eqnarray*}
Condition \eqref{eq33} means the force acting by the wedge does not do work to the gas. Physically, $-(w_p^1, w_p^2)$ represents the limiting force (per unit area) of lift and drag to the wedge in the flow.

For $\epsilon>0$, by Definition \ref{def21}, one can easily check ({\it cf.} \eqref{eq222}) that
$\varrho=\rho^\epsilon(x,y)\LL^2, u=u^\epsilon, v=v^\epsilon, E=E_0$ is a measure solution to  problem \eqref{eq22}-\eqref{eq24}\eqref{eq27}, just by taking
\begin{eqnarray}\label{eq313}\begin{cases}
m^0=m^0(\epsilon)=\rho^\epsilon u^\epsilon\LL^2,\quad n^0=n^0(\epsilon)=\rho^\epsilon v^\epsilon\LL^2,\\
m^1=m^1(\epsilon)=\rho^\epsilon (u^\epsilon)^2\LL^2,\quad n^1=n^1(\epsilon)=\rho^\epsilon u^\epsilon v^\epsilon\LL^2,\\
m^2=m^2(\epsilon)=\rho^\epsilon v^\epsilon u^\epsilon\LL^2,\quad n^2=n^2(\epsilon)=\rho^\epsilon (v^\epsilon)^2\LL^2,\\
m^3=m^3(\epsilon)=\rho^\epsilon u^\epsilon E_0\LL^2,\quad n^3=n^3(\epsilon)=\rho^\epsilon v^\epsilon E_0\LL^2,\\
\wp=\wp(\epsilon)=p^\epsilon\LL^2,\quad w^1_p=w^1_p(\epsilon)=-\frac{a}{\sqrt{1+a^2}}p_1^\epsilon,\quad w^2_p=w^2_p(\epsilon)=\frac{1}{\sqrt{1+a^2}}p_1^\epsilon.
\end{cases}\end{eqnarray}

\subsection{A measure solution for hypersonic limit case $\epsilon=0$}
\label{sec32}

We now construct a measure solution to  problem \eqref{eq22}-\eqref{eq24}\eqref{eq27} for the case $\epsilon=0$.

Let $\textsf{I}_A$ be the characteristic function of a set $A$ ({\it i.e.} $\textsf{I}_A(x,y)=1$ for $(x,y)\in A$ and $=0$ otherwise).  Suppose that
\begin{eqnarray}
&&m^0=\rho_0u_0\ti\LL^2+w^0_m(x)\delta_W=\ti\LL^2+w^0_m(x)\delta_W,\label{eq314}\\ &&n^0=\rho_0v_0\ti\LL^2+w^0_n(x)\delta_W=w^0_n(x)\delta_W,\label{eq315}
\end{eqnarray}
where $w^0_m(x)$, $w^0_n(x)$ are functions to be determined.
Substituting these into \eqref{eq34}, one has
\begin{multline*}
\int_\Omega\p_x\phi\,\dd x\dd y+\sqrt{1+a^2}\left(\int_0^\infty w^0_m(x)\p_x\phi(x,ax)\,\dd x+\int_0^\infty w^0_n(x)\p_y\phi(x,ax)\,\dd x\right)\\+\int_0^\infty\phi(0,y)\,\dd y=0.
\end{multline*}
Since $\p_x\phi(x,ax)=\frac{\dd}{\dd x}\phi(x,ax)-a\p_y\phi(x,ax)$, and using Green theorem, it follows that
\begin{multline*}
-w^0_m(0)\phi(0,0)-\int_0^\infty \frac{\dd}{\dd x}w^0_m(x)\phi(x,ax)\,\dd x+\int_0^\infty (-aw^0_m(x)+w^0_n(x))\p_y\phi(x,ax)\,\dd x\\+\frac{a}{\sqrt{1+a^2}}\int_0^\infty\phi(x,ax)\,\dd x=0.
\end{multline*}
By arbitrariness of $\phi$, this implies that
\begin{eqnarray*}
&&-aw^0_m(x)+w^0_n(x)=0,\quad w^0_m(0)=0;\\
&&\frac{\dd}{\dd x}w^0_m(x)=\frac{a}{\sqrt{1+a^2}}.
\end{eqnarray*}
We then solve that
\begin{eqnarray}\label{eq316}
w^0_m(x)=x\sin\theta, \quad w^0_n(x)=x\sin^2\theta/\cos\theta.
\end{eqnarray}
Similarly we may obtain that
\begin{eqnarray}
&&m^3=E_0\ti\LL^2+(E_0\sin\theta)x\delta_W,\label{eq317}\\ &&n^3=(E_0\sin^2\theta/\cos\theta)x\delta_W.\label{eq318}
\end{eqnarray}

Next we take
\begin{eqnarray}
m^1=\ti\LL^2+w^1_m(x)\delta_W,\quad n^1=w^1_n(x)\delta_W,\quad \wp=0,\label{eq319}
\end{eqnarray}
with $w^1_m(0)=0$. Note that $p_0=0$ if $\epsilon=0$, then from \eqref{eq35} we get
\begin{multline*}
\int_0^\infty\Big(w^1_m(x)-\frac1a w^1_n(x)\Big)\p_x\phi(x,ax)\,\dd x+\int_0^\infty\Big(w^1_p+\frac{a}{\sqrt{1+a^2}}-\frac1a\frac{\dd}{\dd x}w^1_n(x)\Big)\phi(x,ax)\,\dd x=0.
\end{multline*}
This requires that
\begin{eqnarray}\label{eq320}
&&\frac{\dd}{\dd x}w^1_n(x)=aw^1_p+\frac{a^2}{\sqrt{1+a^2}}, \quad w^1_n(0)=0;\\
&&w^1_m(x)=\frac1a w^1_n(x),\quad x\ge0.\label{eq321}
\end{eqnarray}

Then we take
\begin{eqnarray}
m^2=w^2_m(x)\delta_W,\quad n^2=w^2_n(x)\delta_W, \label{eq322}
\end{eqnarray}
with $w^2_m(0)=0$. From \eqref{eq36}, recall $v_0=0$, we have
\begin{equation*}
\int_0^\infty\Big(w^2_m(x)-\frac1aw^2_n(x)\Big)\p_x\phi(x,ax)\,\dd x+\int_0^\infty\Big(w^2_p-\frac1a\frac{\dd}{\dd x}w^2_n(x)\Big)\phi(x,ax)\,\dd x=0,
\end{equation*}
and then
\begin{eqnarray*}
&&\frac{\dd}{\dd x}w^2_n(x)=aw^2_p=-w^1_p,\qquad w^2_n(0)=0;\\
&&w^2_n(x)=aw^2_m(x),\qquad x\ge0.
\end{eqnarray*}

Now let $h(x)=\int_0^x w^1_p(t)\,\dd t$. Then
$$w^2_n(x)=-h(x), \quad w^1_n(x)=ah(x)+\frac{a^2}{\sqrt{1+a^2}}x.$$
By \eqref{eq39}\eqref{eq310},
\begin{eqnarray}\label{eq323}
u|_{W}=\frac{ah(x)+\frac{a^2}{\sqrt{1+a^2}}x}{ax\sin\theta},\quad v|_{W}=-\frac{h(x)}{ax\sin\theta}.
\end{eqnarray}
Furthermore, if the measure $\varrho$ exists, then \eqref{eq316} implies the boundary condition \eqref{eq24}. Therefore we get
$h(x)=-(\sin^3\theta) x,$ hence
\begin{eqnarray}\label{eq324}
w^1_p(x)=-\sin^3\theta,\qquad w^2_p(x)=\sin^2\theta\cos\theta, \end{eqnarray}
and
\begin{eqnarray}\label{eq325add}
w^1_n(x)=(\sin^2\theta\cos\theta )x,\qquad w^2_n(x)=(\sin^3\theta)x.
\end{eqnarray}
From \eqref{eq323}\eqref{eq39}\eqref{eq310}, one has
\begin{eqnarray}\label{eq325}
u=\ti+(\cos^2\theta) \tiw,\qquad v=(\sin\theta\cos\theta) \tiw. \end{eqnarray}
Note that $\sqrt{(w^1_p)^2+(w^2_p)^2}=\sin^2\theta$, one may take
\begin{eqnarray}\label{eq326}
p=(\sin^2\theta)\tiw
\end{eqnarray}
as the pressure.
These results coincide with \eqref{eq223}\eqref{eq224}.

Finally we may determine the measure of density
\begin{eqnarray}\label{eq327}
\varrho=\ti\LL^2+\frac{\sin\theta}{\cos^2\theta}x\delta_W
\end{eqnarray}
by \eqref{eq38}.  This, however, cannot be obtained directly from the elementary analysis in \S \ref{sec23}.

\section{Convergence to singular measure solution in hypersonic limit}\label{sec4}

We now show that, as $\epsilon\to0$, the measures defined in \eqref{eq313} converge vaguely to the corresponding measures defined in \eqref{eq314}\eqref{eq315}\eqref{eq317}\eqref{eq318}\eqref{eq319}\eqref{eq322}.

To this end, 
note that for any $\phi\in C_0(\R)$, and recall $\eta=x/y$,
\begin{eqnarray}\label{eq41}
&&\int_0^{\frac{1}{a}}G(V^\epsilon(\eta))\phi(\eta)\,\dd\eta\nonumber\\
&=&G(V(0))\int_0^{1/\sigma^\epsilon}\phi(\eta)\,\dd\eta+G(V^\epsilon_1)\int_{1/\sigma^\epsilon}^{1/a}\phi(\eta)\,\dd\eta \nonumber\\
&=&\underbrace{G(V(0))\int_0^{1/\sigma^\epsilon}\phi(\eta)\,\dd\eta}_{A_\epsilon}
+\underbrace{\left(G(V^\epsilon_1)(\frac{1}{a}-\frac{1}{\sigma^\epsilon})\right)}_{B_\epsilon}
\underbrace{\left(\frac{1}{\frac{1}{a}-\frac{1}{\sigma^\epsilon}}\int_{1/\sigma^\epsilon}^{1/a}\phi(\eta)\,\dd\eta\right)}_{C_\epsilon}.
\end{eqnarray}
Then \eqref{eq226} implies that, as $\epsilon\to 0$,
\begin{eqnarray*}
A_\epsilon\to G(V(0))\int_0^{1/a}\phi(\eta)\,\dd\eta,\quad
C_\epsilon\to \phi(\frac1a),
\end{eqnarray*}
while by \eqref{eq227},
\begin{eqnarray*}
B_\epsilon&=&\frac{1}{a\sigma^\epsilon}\rho_1^\epsilon(\sigma^\epsilon-a)(v_1^\epsilon, u_1^\epsilon v_1^\epsilon, {v_1^\epsilon}^2+p_1^\epsilon/\rho_1^\epsilon, v_1^\epsilon E_0)^\top\\
&\rightarrow& (1, \cos^2\theta, \cos\theta\sin\theta, E_0)^\top.
\end{eqnarray*}
Therefore we proved that
\begin{eqnarray}\label{eq42}
&&\lim_{\epsilon\to0}\int_0^{\frac{1}{a}}G(V^\epsilon(\eta))\phi(\eta)\,\dd\eta\nonumber\\
&=&G(V(0))\int_0^{1/a}\phi(\eta)\,\dd\eta+ \phi(\frac1a)(1, \cos^2\theta, \cos\theta\sin\theta, E_0)^\top.
\end{eqnarray}

Now return to the $(x,y)$-plane. Let $\psi(x,y)\in C_0(\R^2)$ be a test function, and $\phi(\eta, y)=\psi(\eta y, y)$. Then by change-of-variables and Lebesgue dominant convergence theorem,
\begin{eqnarray}\label{eq43add}
&&\lim_{\epsilon\to0}\int_\Omega G(U^\epsilon(x,y))\psi(x,y)\,\dd x\dd y=\int_0^\infty y\lim_{\epsilon\to0} \int_0^{1/a}G(V^\epsilon(\eta))\phi(\eta,y)\,\dd\eta\,\dd y\nonumber\\
&=&\int_0^{\infty} G(V(0))\int_0^{1/a}\phi(\eta,y)y\,\dd\eta\dd y+(1, \cos^2\theta, \cos\theta\sin\theta, E_0)^\top\int_0^\infty \phi(\frac1a, y)y\,\dd y\nonumber\\
&=&\int_\Omega G(V(0))\psi(x,y)\,\dd x\dd y+(1, \cos^2\theta, \cos\theta\sin\theta, E_0)^\top\int_0^\infty a^2x\psi(x,ax)\,\dd x\nonumber\\
&=&\langle G(V_0)\ti\LL^2, \psi\rangle+\langle w_n(x)\delta_W, \psi\rangle,
\end{eqnarray}
where \begin{eqnarray}\label{eq43}
w_n(x)&=&\frac{a^2x}{\sqrt{1+a^2}}(1, \cos^2\theta, \cos\theta\sin\theta, E_0)^\top\nonumber\\&=&\frac{\sin^2\theta}{\cos\theta}x(1, \cos^2\theta, \cos\theta\sin\theta, E_0)^\top,
\end{eqnarray}
which are exactly $(w^0_n(x), w^1_n(x), w^2_n(x), w^3_n(x))^\top$, the weights of measures $(n^0, n^1, n^2, n^3)^\top$ we  calculated in \S \ref{sec32}.

Recall that
\begin{eqnarray}\label{eq44add}
p^\epsilon(x,y)=\begin{cases} p_0^\epsilon=\frac{\epsilon}{\epsilon+1}E'_0,& 0\le \frac{x}{y}<\frac{1}{\sigma^\epsilon},\\
p_1^\epsilon, &\frac{1}{\sigma^\epsilon}<\frac{x}{y}<\frac1a,
\end{cases}
\end{eqnarray}
so
\begin{eqnarray}\label{eq44}
\lim_{\epsilon\to 0}p^\epsilon=0
\end{eqnarray}
point-wise and also in the sense of vague convergence of measures (if we identify $p^\epsilon$ with $p^\epsilon\ti\LL^2$). Then
\begin{eqnarray*}
G(V_0)=(0, 0, 0, 0)^\top
\end{eqnarray*}
and we proved
\begin{eqnarray}\label{eq45}
(n^0(\epsilon), n^1(\epsilon), n^2(\epsilon), n^3(\epsilon))\to (n^0, n^1, n^2, n^3)
\end{eqnarray}
vaguely by \eqref{eq43}.

Similarly, replacing $G$ in \eqref{eq41} by $F$, we can prove that
\begin{eqnarray}\label{eq47}
(m^0(\epsilon), m^1(\epsilon), m^2(\epsilon), m^3(\epsilon))\to (m^0, m^1, m^2, m^3)
\end{eqnarray}
in the sense of vague convergence of measures.

It is now easy to show consistency of \eqref{eq33}-\eqref{eq37} in Definition \ref{def32}.  Using \eqref{eq313}, we may rewrite \eqref{eq211} as
\begin{eqnarray}
\langle m^0(\epsilon), \p_x\phi\rangle+ \langle n^0(\epsilon), \p_y\phi\rangle+\int_0^\infty\phi(0,y)\,\dd y=0,&&\label{eq48}\\
\langle m^1(\epsilon), \p_x\phi\rangle+ \langle n^1(\epsilon), \p_y\phi\rangle+\langle \wp(\epsilon), \p_x\phi\rangle\qquad\qquad&&\nonumber\\-\int_0^\infty \frac{a}{\sqrt{1+a^2}}p_1^\epsilon\phi(x,ax)\sqrt{1+a^2}\,\dd x+\int_0^\infty(1+p^\epsilon_0)\phi(0,y)\,\dd y=0,&&\label{eq49}\\
\langle m^2(\epsilon), \p_x\phi\rangle+ \langle n^2(\epsilon), \p_y\phi\rangle+\langle \wp(\epsilon), \p_y\phi\rangle\qquad\qquad&&\nonumber\\-\int_0^\infty\frac{-1}{\sqrt{1+a^2}}p_1^\epsilon\phi(x,ax)\sqrt{1+a^2}\,\dd x=0,&&\label{eq410}\\
\langle m^3(\epsilon), \p_x\phi\rangle+ \langle n^3(\epsilon), \p_y\phi\rangle+\int_0^\infty E_0\phi(0,y)\,\dd y=0,&&\label{eq411}
\end{eqnarray}
where $\langle\wp(\epsilon),\phi\rangle=\int_\Omega p^\epsilon \phi\,\dd x\dd y$, and $p^\epsilon$ is given by \eqref{eq44add}. Now as $\epsilon\to0,$ recall that $\langle\wp(\epsilon),\phi\rangle\to0$, $p_0^\epsilon\to0$, and $p_1^\epsilon\to\sin^2\theta$, then
\begin{eqnarray}
-\int_0^\infty \frac{a}{\sqrt{1+a^2}}p_1^\epsilon\phi(x,ax)\sqrt{1+a^2}\,\dd x\to\langle w^1_p\delta_W, \phi\rangle,&&\\
\int_0^\infty\frac{1}{\sqrt{1+a^2}}p_1^\epsilon\phi(x,ax)\sqrt{1+a^2}\,\dd x\to\langle w^2_p\delta_W, \phi\rangle,&&
\end{eqnarray}
with the same $w^1_p, w^2_p$ as defined in \eqref{eq324}. Thanks to \eqref{eq45}\eqref{eq47}, we then have \eqref{eq34}-\eqref{eq37} as $\epsilon\to 0$.

We summarize the main results of this paper as the following theorem.

\begin{theorem}\label{thm41}
The problem \eqref{eq22}-\eqref{eq24}\eqref{eq27} has measure solutions for each $\epsilon=\gamma-1\ge0$ in the sense of Definition \ref{def32}, and these solutions converge vaguely as $\epsilon\to 0$ (or Mach number of the upcoming flow goes to infinity) to a singular measure solution with density containing a weighted Dirac measure on the surface of the wedge, given by  \eqref{eq314}\eqref{eq315}\eqref{eq317}\eqref{eq318}\eqref{eq319}\eqref{eq322} or \eqref{eq325}-\eqref{eq327}.
\end{theorem}

\section*{Acknowledgements}
The research of Aifang Qu is supported by National Natural Science Foundation of China under Grant No. 11571357.
The research of Hairong
Yuan (the corresponding author) is supported by National Natural Science Foundation of China under Grant No. 11371141, and by Science and Technology Commission of Shanghai Municipality (STCSM) under grant No. 18dz2271000.

\end{document}